\documentclass[12pt]{amsart}

\usepackage{geometry}

\usepackage{color}

\numberwithin{equation}{section} 
\numberwithin{figure}{section} 
\theoremstyle{plain}
\newtheorem{thm}{Theorem}[section]
\theoremstyle{plain}
\newtheorem{lem}[thm]{Lemma}
\theoremstyle{plain}
\newtheorem{prop}[thm]{Proposition}
\theoremstyle{plain}
\newtheorem{cor}[thm]{Corollary}
\theoremstyle{remark}
\newtheorem{rem}[thm]{Remark}
\theoremstyle{remark}
\newtheorem{note}[thm]{Notation}
\theoremstyle{remark}
\newtheorem{example}[thm]{Example}
\theoremstyle{plain}
\newtheorem*{prop*}{Proposition}

\theoremstyle{remark}
\newtheorem*{rem*}{Remark}

\theoremstyle{plain}
\newtheorem{question}{Question}

\theoremstyle{remark}

\theoremstyle{definition}

\theoremstyle{definition}
\newtheorem{defi}[thm]{Definition}

\usepackage{latexsym,amsfonts, xypic,amssymb}
\xyoption{all}
\newcommand{\bc}{\mathbb{C}}
\newcommand{\C}{\mathbb{C}}
\newcommand{\R}{\mathbb{R}}

\newcommand{\N}{\mathbb{N}}

\newcommand{\ba}{\mathbb{A}}

\newcommand{\LND}{\textup{LND}}
\newcommand{\ML}{\textup{ML}}
\newcommand{\Aut}{\textup{Aut}}
\newcommand{\Ker}{\textup{Ker}}
\newcommand{\id}{\textup{id}}
\newcommand{\diff}{\textup{d}}

\begin{document}

\title{Inequivalent embeddings of the Koras-Russell cubic threefold}

\author{A. Dubouloz, L. Moser-Jauslin \and P.-M. Poloni}
\address{Universite de Bourgogne ,
Institut de Math\'ematiques de Bourgogne CNRS -- UMR 5584,
9~avenue Alain Savary - B.P. 47 870, 21078 Dijon, France}
\email{Adrien.Dubouloz@u-bourgogne.fr
 }
\email{lucy.moser-jauslin@u-bourgogne.fr }
\email{pierre-marie.poloni@u-bourgogne.fr }

\begin{abstract} Let $X$ be the hypersurface of the complex affine four-space $\mathbb{a}^4$ defined by the equation $x^2y+z^2+t^3+x=0$. It is well-known that $X$ is an affine contractible smooth threefold, which is not algebraically isomorphic to affine three-space. The main result of this article is to show that there exists another hypersurface $Y$ of $\mathbb{a}^4$, which is isomorphic to $X$, but such that there exists no automorphism of the ambient four-space which restricts to  an isomorphism between $X$ and $Y$. In other words, the two hypersurfaces are inequivalent. To prove this result, we give a description of the automorphism group of $X$. We show  that all algebraic automorphisms of $X$ extend to automorphisms of $\ba^4$. As a corollary, we find that every automorphism of $X$ fixes the point $(0,0,0,0)\in X$.
\end{abstract}

\maketitle

\section{Introduction}

The Koras-Russell cubic threefold is the hypersurface $X$ of the
complex affine space $\mathbb{A}^{4}={\rm
Spec}\left(\mathbb{C}\left[x,y,z,t\right]\right)$ defined by the
equation \[ P=x+x^{2}y+z^{2}+t^{3}=0.\]

In this paper we will study certain properties of this threefold.
The point of view comes from the elementary remark that the
Koras-Russell threefold can be interpreted as a one-parameter
family of Danielewski hypersurfaces. A \emph{Danielewski
hypersurface} is a subvariety of $\mathbb{A}^{3}={\rm
Spec}\left(\mathbb{C}\left[x,y,z\right]\right)$ defined by an
equation of the form $x^ny=q(x,z)$, where $n$ is a non-zero
natural number and $q(x,z)\in\bc[x,z]$ is a polynomial such that
$q(0,z)$ is of degree at least two. Such hypersurfaces have been
studied by the authors in  \cite{D-P}, \cite{MJ-P}  and
\cite{Poloni}. This interpretation allows us to deduce results
similar to the ones for Danielewski hypersurfaces for this threefold.

An important question in affine algebraic geometry asks whether
every embedding of complex affine $k$-space $\ba^k$  in $\ba^n$,
where $k<n$, is rectifiable, i.e., is equivalent to an embedding
as a linear subspace.  The Abyhankar-Moh-Suzuki theorem shows that
the answer   is ``yes" if $n=2$ (\cite{A-M}, \cite{Suzuki}),  and,
by a general result proved independantly by Kaliman \cite{Kaliman}
and Srinivas \cite{Srinivas}, if $n\ge 2k+2$ the answer is also
affirmative. However, all other cases remain open.

Here we are interested in the case of embeddings of hypersurfaces.
It is easy to find affine varieties of dimension $n$ admitting
non-equivalent embeddings into $\ba^{n+1}$.  For example, the
punctured line $\mathbb{A}^1\setminus\{0\}$ has many
non-equivalent embeddings in $\ba^2$. For each $n\in\N$, let
$P_n=x^ny-1$. The subvariety defined by the zero set of $P_n$ is
isomorphic to $\mathbb{A}^1\setminus\{0\}$, however, the induced
embeddings for each $n$ are inequivalent. This can be seen by the
fact that the subschemes defined by $P_n-1=0$ are all
non-isomorphic.  It is more difficult to find examples where all
non-zero fibers of the defining polynomial are irreducible.

 In example 6.3 of \cite{K-Z}, Kaliman and Zaidenberg gave examples of acyclic surfaces which admit non-equivalent embeddings in three-space. In these cases, the obstruction for equivalence is essentially topological, the non-zero fibers of the polynomials which define the two hypersurfaces being non homeomorphic. Even if they are contractible,  these surfaces are algebraically remote from the affine plane due to the fact that they have nonnegative logarithmic Kodaira dimension. In contrast, an example is given \cite{F-MJ} for Danielewski hypersurfaces, where the non-zero fibers are algebraically non-isomorphic, but analytically isomorphic, and in \cite{MJ-P} and \cite{Poloni}, examples of Danielewski hypersurfaces are given where all fibers are algebraically isomorphic, but the embeddings are non-equivalent. Danielewski hypersurfaces are rational, whence close to the affine plane from an algebraic point of view,  but they have non-trivial singular homology groups.  However, the techniques used in \emph{loc. cit.} are purely algebraic and do not depend on the topological properties of these surfaces. As we shall see here, similar ideas can be used to treat the case of a variety diffeomorphic to $\R^6$. For other inequivalent embeddings of hypersurfaces, see, for example, \cite{S-Y}.

In the present article, we use similar techniques as for the
Danielewski hypersurfaces to study the Koras-Russell threefold $X$
above. For a polynomial $f\in\C[x,y,z,t]$, we denote by $V(f)$ the
subscheme of $\ba^4$ defined by the zero set of $f$ in $\ba^4$,
that is, $V(f)={\rm Spec}(\mathbb{C}[x,y,z,t]/(f))$.  The
hypersurface $X=V(P)$ is smooth and contractible, and is therefore
diffeomorphic to $\ba^3$ \cite{Dimca}. However, it was shown by
Makar-Limanov that it is not isomorphic to affine three-space
\cite{ML}. We show that there is another hypersurface $Y=V(Q)$
which is isomorphic to $X$, but there is no algebraic automorphism
of four-space which restricts to an isomorphism between $X$ and
$Y$. Thus we have at least two inequivalent embeddings of $X$. For
this example, the two hypersurfaces are analytically equivalent by
a holomorphic automorphism which preserves the fibers of $P$ and
$Q$, and,  therefore, as for certain examples of Danielewski
hypersurfaces, there is no topological obstruction to the
existence of such an automorphism. In other words, the obstruction
to extending automorphisms in this case is purely algebraic. Also,
for all $c\in\C\setminus\{1\}$, the fibers $V(P+c)$ and $V(Q+c)$
are isomorphic.  It is an open question if $V(P+1)\cong V(Q+1)$.

The methods to study the question of equivalent embeddings are similar to those used in the articles cited above for Danielewski hypersurfaces. However, they must be adapted in order to consider a higher dimensional variety. They are based on certain properties of the automorphism group of the varieties. 
The set of locally nilpotent derivations on a Danielewski hypersurface is explicitly known (see \cite{D-P}; see also \cite{ML2001}), and the Makar-Limanov invariant is non trivial when $n\ge 2$. This restricts the possibilities for automorphisms of these surfaces. The Makar-Limanov invariant of the Koras-Russell threefold is also non trivial \cite{ML}. In section 3, we determine the complete automorphism group of $X$. For this case, new methods are needed, since the restrictions given by the Makar-Limanov invariant do not suffice to determine the automorphism group. 
 As a corollary, we find the surprising result that any algebraic automorphism of $X$ fixes the point $(0,0,0,0)\in X$. (See corollary \ref{fix-origin}.) Also, all automorphisms of $X$ extend to automorphisms of $\ba^4$.

{\bf Acknowledgements :} We would like to thank Gene Freudenburg
and St\'ephane V\'en\'ereau for helpful discussions concerning the
automorphism group of $X$.

\section{Inequivalent embeddings in $\mathbb{A}^{4}$ }

We  denote by $P$ and $Q$ the following polynomials of
$\C^{[4]}=\C[x,y,z,t]$:

$$P=x^2y+z^2+x+t^3 \quad \text{ and } \quad Q=x^{2}y+\left(1+x\right)\left(z^{2}+x+t^{3}\right).$$
The Koras-Russell cubic threefold is the hypersurface
$X\subset\mathbb{A}^{4}={\rm
Spec}\left(\mathbb{C}\left[x,y,z,t\right]\right)$ defined by the
equation $P=0$ whereas  $Y$ denote the hypersurface defined by the
equation $Q=0$.

We start by giving  some definitions to clarify the difference
between {\it inequivalent embeddings} and {\it inequivalent
hypersurfaces}.

\begin{defi}
A regular map $\phi:X\to Z$ between two algebraic varieties is a
{\em closed embedding of $X$ in $Z$} if:
\begin{enumerate}
\item $\phi(X)$ is a closed subvariety of $Z$; \item
$\phi:X\to\phi(X)$ is an isomorphism.
\end{enumerate}
\end{defi}

\begin{defi}
Two embeddings $\phi_1,\phi_2:X\to Z$ are {\em equivalent} if
there exists an automorphism $\Psi$ of $Z$ such that
$\phi_2=\Psi\circ\phi_1$.
\end{defi}

\begin{defi} Two subvarieties $X_1$ and $X_2$ of $Z$ are {\em equivalent} if there exists an automorphism $\Psi$ of $Z$ such that $\Psi(X_1)=X_2$.  If $X_1$ and $X_2$ are hypersurfaces, then we say they are {\it equivalent hypersurfaces}.
\end{defi}

In this article, we will show that $V(P)$ and $V(Q)$ are
isomorphic as abstract threefolds, however, they are inequivalent
as hypersurfaces of $\ba^4$, in the sense of  definition 2.3. In
other words, no isomorphism between $V(P)$ and $V(Q)$ extends to
an automorphism of $\ba^4$. We will do this in two steps. First,
we will find an isomorphism $\phi$ between $V(P)$ and $V(Q)$ which
does not extend to an automorphism of $\ba^4$. This implies that
the two embeddings $i_1:V(P)\to\ba^4$ and
$i_2\circ\phi:V(P)\to\ba^4$ are inequivalent in the sense of
definition 2.2. (Here, $V(P)$ is the scheme defined as ${\rm
Spec}(\C[x,y,z,t]/(P))$, and $i_1$ is the embedding of $V(P)$ in
$\ba^4$ corresponding to the canonical homomorphism
$i_1^*:\C[x,y,z,t]\to\C[x,y,z,t]/(P)$. The embedding $i_2$ is
defined similarly, by replacing $P$ by $Q$.)

The second step, discussed in section 3,  will be to study the
automorphism group of $V(P)$. We will show in section 4 that all
automorphisms of this hypersurface extend to automorphisms of
$\ba^4$.

Finally, putting these two results together, we will show the stronger
result that $V(P)$ and $V(Q)$ are inequivalent hypersurfaces.

In this article, we use the following key result  concerning the
Makar-Limanov invariant of an irreducible affine variety. Given an
irreducible affine variety $Z$, denote by $\C[Z]$ the ring of
regular functions on $Z$. A {\it locally nilpotent derivation} of
$\C[Z]$  is a $\C$-derivation $\partial:\C[Z]\to\C[Z]$ such that
for any $f\in\C[Z]$, there exists $n\in\N$ such that
$\partial^n(f)=0$. If $\partial$ is a locally nilpotent derivation
of $\C[Z]$, then $\exp(\lambda\partial)$ defines an algebraic
action of $(\C,+)$ on $Z$, and all $(\C,+)$ actions arise in this
way. The kernel of a locally nilpotent derivation  is the
subalgebra of invariants in $\C[Z]$ of the corresponding action.
The {\it Makar-Limanov invariant}, $\ML(\C[Z])$ is defined as the
intersection of  the kernels of all locally nilpotent derivations
on $\C[Z]$. If $Z$ is affine space, the Makar-Limanov invariant is
simply $\C$. However, it was shown in \cite{ML} that
$\ML(\C[X])=\C[x]\not=\C$. This implies that the Koras-Russell threefold $X$ is not isomorphic
to the affine three-space. In  theorem 9.1 and example 9.1 of
\cite{K-ML}, the result of Makar-Limanov is generalized, and it is
shown  in particular that for every $c\in\C$, and every $\lambda\in\C^*$,
$$\ML\left(\C[x,y,z,t]/\left(Q-c\right)\right)=\ML\left(\C[x,y,z,t]/\left(\lambda P-c\right)\right)=\C[x].$$

We will use this result throughout the article.

\begin{thm}\label{main thm}The
following hold :
\begin{itemize}
\item[a)] The endomorphism of $\mathbb{A}^{4}$ defined by
$(x,y,z,t)\mapsto (x,(1+x)y,z,t)$ induces an isomorphism
$\phi:X\stackrel{\sim}{\rightarrow} Y$.

\item[b)] $\phi$ cannot be extended to an algebraic automorphism of
$\mathbb{A}^{4}$.
\end{itemize}
\end{thm}

\begin{proof}
Let $\Phi$ and $\Psi$ be the endomorphisms of $\mathbb{A}^{4}$
defined
 by $$\begin{cases}
\Phi: & \left(x,y,z,t\right)\mapsto\left(x,\left(1+x\right)y,z,t\right)\\
\Psi: &
\left(x,y,z,t\right)\mapsto\left(x,\left(1-x\right)y-x-z^2-t^3,z,t\right)\end{cases}$$
Denote by $\Phi^*$  the endomorphism of $\C[x,y,z,t]$
corresponding to $\Phi$ ($\Phi^*$ fixes $x$, $z$ and $t$, and
$\Phi^*(y)=(1+x)y$), and  by $\Psi^*$ the corresponding
endomorphism to $\Psi$. One checks that $\Phi^*(Q)=(1+x)P$ and
$\Psi^*(P)=(1-x)Q$. Thus, $\Phi$ induces a well-defined  morphism
of algebraic varieties $\phi:X\to Y$, and $\Psi$ induces a
well-defined regular map $\psi:Y\to X$.

Since $\phi$ and $\psi$ are morphisms of schemes over
$\mathbb{A}^{3}={\rm Spec}\left(\C[x,z,t]\right)$, the identities
$(\Phi^*\circ\Psi^*)(y)=y-P $ and $(\Psi^*\circ\Phi^*)(y)=y-Q $
guarantee that they are inverse isomorphisms.

The second assertion follows from the remark  that
$\phi((0,0,0,0))=(0,0,0,0)$, and from the first part of
proposition \ref{prop:extension} below.
\end{proof}

\begin{prop}\label{prop:extension} The following results hold.
\begin{itemize}
\item[(i)] Suppose that there exists $\Xi$  an algebraic
automorphism of $\mathbb{A}^{4}$ which restricts to an isomorphism
 between $X$ and $Y$. Then $\Xi$ does not fix the point $(0,0,0,0)$, i.e.
$\Xi\left(0,0,0,0\right)\neq\left(0,0,0,0\right)$. \item[(ii)] If
$\Xi$ is an algebraic automorphism of $\mathbb{A}^{4}$ which
restricts to an automorphism of $X$, then $\Xi$ fixes the point
$(0,0,0,0)$. \item[(iii)] If $\Xi$ is an algebraic automorphism of
$\mathbb{A}^{4}$ which restricts to an automorphism of $Y$, then
$\Xi$ fixes the point $(0,-1,0,0)$.
\end{itemize}
\end{prop}

\begin{proof}
For part (i), suppose that $\Xi$ is an automorphism of $\ba^4$ which
extends an isomorphism between $X$ and $Y$.

 Note that since
$Q=x^2y+(1+x)(z^2+x+t^3)$ is irreducible, it follows that
$\Xi^*(Q)=\lambda P$ for a certain $\lambda\in\C^*$. Thus, $\Xi$
maps $V(P-\lambda^{-1}c)$ isomorphically onto $V(Q-c)$ for every
$c\in\C$.

Let us show first that $\Xi^*(x)=\lambda x$. As mentioned above,
for any $c\in\C$,  the Makar-Limanov invariant of $\left(\C[x,y,z,t]/\left(Q-c\right)\right)$ and
 of $\left(\C[x,y,z,t]/\left(\lambda P-c\right)\right)$ are  $\C[x].$
This implies that for every $c\in\mathbb{C}$, $\Xi^{*}$ restricts
to an isomorphism
\[\mathbb{C}\left[x\right]=\ML\left(\mathbb{C}\left[x,y,z,t\right]/\left(Q-c\right)\right)\stackrel{\sim}{\rightarrow}\ML\left(\mathbb{C}\left[x,y,z,t\right]/\left(\lambda
P-c\right)\right)=\mathbb{C}\left[x\right].\] Combined with the
fact that for any $c\in\C$, $V(x-a,Q-c)$ and $V(x-a,P-c)$ are
isomorphic to $\ba^2$ if and only if $a\not=0$, we find that
$\Xi^*$ preserves the ideal $(x)$. 

Let $\Xi^*(x)=\mu x$ with $\mu\in\C^*$. Then $\Xi^{*}\left(Q-x\right)=\lambda P-\mu
x$. One checks easily that the variety $Z={\rm
Spec}\left(\mathbb{C}\left[x,y,z,t\right]/\left(Q-x\right)\right)$
is singular along the line $L_{z}=\left\{ x=z=t=0\right\} $ in
$\mathbb{A}^{4}$. On the other hand, it follows from the Jacobian
criterion that the variety $\Xi^{-1}\left(Z\right)=V\left(\lambda
P-\mu x\right)$ is singular if and only if $\lambda=\mu$. This
implies that $\lambda=\mu$.

In other words, we have shown that  $\Xi^*(Q-x)=\lambda (P-x)$.
Thus, $\Xi$ restricts to an isomorphism between $W_P=V(P-x)$ and
$W_Q=V(Q-x)$.

For parts (ii) and (iii), a similar argument shows that any
automorphism of $\ba^4$ which extends an automorphism of $X$ will
preserve the subvariety $W_P$,  and any automorphism of $\ba^4$
which extends an automorphism of $Y$ will preserve the subvariety
$W_Q$.

Now we look more carefully at the two subvarieties $W_P$ and
$W_Q$. They are both singular along the line.
$\left\{x=z=t=0\right\}$.  We look now at the tangent cone of each
singular point of $W_P$ and $W_Q$. Let $p_0=(0,y_0,0,0,0)$.

 We deduce from the
identity
\begin{eqnarray*} P-x & = &
y_{0}x^{2}+z^{2}+x^{2}\left(y-y_{0}\right)+t^{3}\end{eqnarray*}
that the tangent cone $TC_{p_{0}}(W_P)$ of $W_P$ at $p_{0}$ is
isomorphic to
 ${\rm
Spec}\left(\mathbb{C}\left[x,y,z,t\right]/\left(z^{2}+y_{0}x^{2}\right)\right).$
In particular $TC_{p_{0}}W_P$ consists of two distinct hyperplanes
for $y_{0}\neq 0$ and of the double hyperplane $\left\{
z=0\right\} $ if $y_0=0$. On the other hand, we deduce from the
identity\begin{eqnarray*}Q-x & = &
(y_{0}+1)x^{2}+z^{2}+x^{2}\left(y-y_{0}\right)+t^{3}+xz^{2}+xt^{3}\end{eqnarray*}
that the tangent cone of $W_Q$ at a  point
$p_{0}=\left(0,y_{0},0,0\right)$ is isomorphic to \hfill\break
 ${\rm
Spec}\left(\mathbb{C}\left[x,y,z,t\right]/\left(z^{2}+(y_{0}+1)x^{2}\right)\right)$.
Thus the tangent cone $TC_{p_{0}}(W_Q)$ consists of two distinct
hyperplanes for $y_{0}\neq -1$ and of the double hyperplane
$\left\{ z=0\right\} $ if $y_0=-1$.

For part (i), since $\Xi(W_P)=W_Q$, we have that
$\Xi(0,0,0,0)=(0,-1,0,0)$. For part (ii), any automorphism of
$W_P$ fixes the point $(0,0,0,0)$. Finally, for part (iii), any
automorphism of $W_Q$ fixes the point $(0,-1,0,0)$. This completes
the proof.
\end{proof}

\begin{cor}
The Koras-Russell cubic threefold admits at least two
non-equivalent embeddings in $\mathbb{A}^{4}$.
\end{cor}

\begin{proof} Consider the inclusions $i_1: X\hookrightarrow\ba^4$ and $i_2:Y\hookrightarrow\ba^4$. The embeddings $i_1$ and $i_2\circ\phi$ are inequivalent embeddings by the theorem \ref{main thm}.
\end{proof}

\section{The automorphism group of  $X$}
\def\A{{\mathcal A}}

We will now determine the structure of the  automorphism group
$\Aut(X)$. We start with some notation. If $S$ is a ring and $R$
is a subring,  then $\Aut_R(S)$ denotes the group of ring
automorphisms of $S$ which fix $R$. Denote by $\C[X]=\C^{[4]}/(P)$
the ring of regular functions on $X$. The group $\Aut(X)$ is
isomorphic to  the group $\Aut(\C[X])=\Aut_\C(\C[X])$.

\begin{note} Denote by $I=(x^2,z^2+t^3+x)\subset\C[x,z,t]$ the ideal generated by $x^2$ and $z^2+t^3+x$. Let $\A$ be the subgroup of $\Aut(\C[x,z,t])$ of automorphisms which preserve the ideals $(x)$ and $I$. Let $\A_1$ be the subgroup of $\A$ of automorphisms $\varphi\in\A$ such that $\varphi$ fixes $x$, and $\varphi\equiv\id\mod (x)$. Finally, let $\A_2$ be the subgroup of $\Aut_{\C[x]}(\C[x][z,t])$ which are equivalent to the identity modulo $x^2$:

$
\begin{array}{rl }
 \A     &=\{\varphi\in\Aut(\C[x,z,t]) \mid \varphi((x))=(x), \varphi(I)=I\};
    \\
 \A_1     &   =\{\varphi\in\Aut_{\C[x]}(\C[x][z,t])\mid\varphi(I)=I, \varphi\equiv\id\mod (x)\};
\\
\A_2&=\{\varphi\in\Aut_{\C[x]}(\C[x][z,t])\mid\varphi\equiv\id\mod
(x^2)\}.

\end{array}
$
\end{note}

It is clear that $\A_2$ is a normal subgroup of $\A_1$, and $\A_1$
is a normal subgroup of $\A$.

The ring $\C[X]$ can be viewed as the subalgebra of
$\C[x,x^{-1},z,t]$, generated by $x$, $z$, $t$ and
$(z^2+t^3+x)/x^2$. In particular, it contains $\C[x,z,t]$ as a
subring.

The following proposition can be deduced from the results of
Makar-Limanov concerning the set of locally nilpotent derivations
on $X$. See \cite{ML} and  \cite{Freudenburg}.

\begin{prop}\label{aut-group} The automorphism group  $\Aut(X)\cong\Aut(\C[X])$ is isomorphic to the group $\A$. The isomorphism of $\Aut(\C[X])$ to $\A$ is induced by restriction of any automorphism of $\C[X]$ to the subalgebra $\C[x,z,t]$.
\end{prop}

\begin{proof}
 In \cite{ML}, it was shown that the Makar-Limanov
invariant of $\C[X]$ is $\C[x]$. In fact, more was proven. It was
shown that for any locally nilpotent derivation $\partial$ of
$\C[X]$, we have that $\ker\partial^2\subset \C[x,z,t]$. There
exists a locally nilpotent derivation $\partial_1$ on $\C[X]$  such
that $\partial_1(z)=0$. Thus for any automorphism $\varphi$ of
$\C[X]$, $\varphi\circ\partial_1\circ\varphi^{-1}$ has $\varphi(z)$
in its kernel. In particular, $\varphi(z)\in\C[x,z,t]$. Also $\varphi^{-1}(z)$ belongs to $\C[x,z,t]$, and therefore, $z=\varphi(\varphi^{-1}(z))$ is in the image $\varphi(\C[x,z,t])$. The same
holds for $t$. Thus we have shown that any automorphism of $\C[X]$
restricts to an automorphism of the subalgebra $\C[x,z,t]$ of
$\C[X]$. 
Also, any automorphism $\varphi\in\C[X]$ stabilizes the ideal $(x)$. Indeed, since the Makar-Limanov invariant $\C[x]\subset\C[X]$  is stable, there exists $\lambda\in\C^*$ and $b\in\C$ such that $\varphi(x)=\lambda x+b$. Also, the zero set of $\lambda x+b$ in $X$ is singular if and only if $b=0$. Thus, since the zero set of $x$ in $X$ is singular, we have that $b=0$, and thus the ideal $(x)$ is preserved.

Now, to show that the ideal $I$ is preserved, we show that
$x^2\C[X]\cap\C[x,z,t]=I$. It is clear that $I$ is contained in
the intersection. For the converse, note that any element in
$\C[X]$ is represented in a unique way as a polynomial of the form
$yf_0(y,z,t)+xyf_1(y,z,t)+g(x,z,t)$, where $f_0,f_1\in\C[y,z,t]$
and $g\in\C[x,z,t]$. (In other words, all monomials containing
$x^2y$ are eliminated.) If we look now at the ideal generated by
$x^2$, $x^2yf_0+x^3yf_1\in\C[x,z,t]$ if and only if $f_0$ and
$f_1$ are independent of $y$, and in this case,
$$x^2yf_0+x^3yf_1=x^2y(f_0+xf_1)=(-z^2-t^3-x)(f_0+xf_1)\in
(z^2+t^3+x)\subset\C[x,z,t].$$

Finally, we show that any automorphism $\varphi$ in $\A$ extends
to a unique automorphism $\tilde{\varphi}$ of $\C[X]$. To prove uniqueness, note that any element of $\A$ induces a unique automorphism of $\C[x,x^{-1},z,t]$, which contains $\C[X]$. To prove existence, we
extend $\varphi$ to an endomorphism $\Phi$ of $\C[x,y,z,t]$, such
that $\Phi(P)$ is in the ideal generated by $P$. More precisely,
we define $\Phi(h)=h$ if $h\in\C[x,z,t]$, and we construct
$\Phi(y)$ as follows. By hypothesis, there exists
$f,g\in\C[x,z,t]$ such that
$\varphi(z^2+t^3+x)=(z^2+t^3+x)f+x^2g$. Suppose also that
$\varphi(x)=\lambda x$, with $\lambda\in\C^*$. Now we pose
$\Phi(y)=(yf-g)/\lambda^2$. One checks easily that $\Phi(P)=fP$.
Thus $\Phi$ induces a homomorphism $\tilde{\varphi}$ from $\C[X]$
to $\C[X]$. We will show it is an automorphism. Let $\psi$ be the
inverse of $\varphi$ as an automorphism of $\C[x,z,t]$. Construct
$\Psi$ and $\tilde{\psi}$ as above. Then
$\tilde{\psi}\circ\tilde{\varphi}$ and
$\tilde{\varphi}\circ\tilde{\psi}$ are homomorphisms of $\C[X]$
which extend to the identity as homomorphisms of
$\C[x,x^{-1},z,t]$. Thus $\tilde{\psi}$ is the inverse of
$\tilde{\varphi}$.
\end{proof}

\begin{rem}
This result has the following geometric interpretation. The
inclusion $\C[x,z,t]\subset\C[X]$ induces a dominant morphism
$\sigma:X\to \ba^3$. Any automorphism of $X$ is the lifting by
$\sigma$ of a unique automorphism of $\ba^3$. More precisely, if
$\tilde{\varphi}$ is an automorphism of $X$, there is a unique
automorphism ${\varphi}$ of $\ba^3$ such that
${\varphi}\circ\sigma=\sigma\circ\tilde{\varphi}$. Also, an
automorphism $\varphi$ of $\ba^3$ has a lifting as an automorphism
of $X$ if and only if ${\varphi}$ preserves the ideals $(x)$ and
$I$.
\end{rem}

We will now discuss the structure of the group $\A$. Note that it
contains a subgroup isomorphic to $\C^*$, (corresponding to a
$\C^*$-action on $X$) given by the $\C^*$-action where $x$ has
weight 6, $z$ has weight 3 and $t$ has weight 2. .
\begin{prop} $\A=\A_1\rtimes\C^*$.
\end{prop}
\begin{proof}
It is clear that $\A_1\rtimes\C^*$ is a subgroup of $\A$. We will
now show that $\A_1$ and $\C^*$ generate $\A$. First note that if
$\varphi\in\A$, then, since $\varphi$ preserves the ideal $(x)$,
it induces an automorphism $\overline\varphi$ of
$\C[x,z,t]/(x)\cong\C[z,t]$. Also, since $I$ is preserved, the
ideal $(z^2+t^3)$ is preserved by $\overline\varphi$.  By
composing with an automorphism in $\C^*$, we can assume that
$\overline\varphi(z^2+t^3)=z^2+t^3$. In particular, for all
$c\in\C$, $\overline\varphi$ induces an automorphism of
$V(z^2+t^3+c)$. If $c\not=0$, this defines a smooth elliptic curve
$E$ with one point $p$ removed. The group of automorphisms of this
affine curve is the group of automorphisms of $E$ which fix the
point $p$. This group is of order 6, generated by the automorphism
that fixes $t$ and sends $z$ to $-z$, and the automorphism that
fixes $z$ and sends $t$ to $e^{i2\pi/3}t$ (see, for example,
\cite{Hartshorne}). There are therefore only 6 automorphisms of
$\C[z,t]$ which fixes the polynomial $z^2+t^3$, and they are all
in the image of $\C^*$. We can therefore suppose that
$\overline\varphi(z)=z$ and $\overline\varphi(t)=t$. This means
exactly that $\varphi\in\A_1$.
\end{proof}

Now we are left with the problem of understanding the group
$\A_1$. For this part, we will consider a more general situation.
First note that the group $\A_1$ is exactly the group of
automorphisms $\varphi$ of $\C[x,z,t]$ which fix $x$, such that
$\varphi\equiv\id\mod (x)$, and such that $\varphi(z^2+t^3)\in
(x^2, z^2+t^3)$.

\begin{note} Let $r\in\C[z,t]$ be a polynomial. Denote by $\A_1(r)$ the group

$$\A_1(r)=\{\varphi\in\Aut_{\C[x]}(\C[x][z,t])\mid\varphi\equiv\id\mod (x), \varphi(r)\in(x^2,r)\}.$$
\end{note}
We  have thus that $\A_1=\A_1(z^2+t^3)$, and that, for any $r$,
$\A_2$ is a normal subgroup of $\A_1(r)$.

We use the following standard notation for partial derivatives. If
$h\in\C[z,t]$, $h_z=\partial h/\partial z$, $h_t=\partial h/
\partial t$, and if $h,f\in\C[z,t]$, then the Poisson bracket of $h$ and $f$ is
given by $\{h,f\}=h_zf_t-h_tf_z$.

\begin{prop} Let $r\in\C[z,t]$ be a polynomial with
no multiple irreducible factor and  such that the zero set
$V(r)\in\ba^2$ is connected. Then,  for every
automorphism there exists $\varphi\in\A_1(r)$, a polynomial $\alpha\in\C[z,t]$
such that $\varphi(z)\equiv z+x(r\alpha)_t\mod (x^2)$ and
$\varphi(t)\equiv t-x(r\alpha)_z\mod (x^2)$.

Moreover, $\theta:\A_1(r)\to\C[z,t]$, $\varphi\mapsto\alpha$ is a
surjective group homomorphism whose kernel is $\A_2$. In particular, the
quotient group $\A_1(r)/\A_2$ is isomorphic to the additive group
$(\C[z,t], +)$.\end{prop}

\begin{proof} For any  $\varphi\in\A_1(r)$, we have
that $\varphi(z)\equiv z+xf \mod (x^2)$ and $\varphi(t)\equiv t+xg
\mod (x^2)$, where $f,g\in\C[z,t]$. By hypothesis, $\varphi$ is an
automorphism. Therefore, its Jacobian equals one. This implies in
particular that $f_z+g_t=0$. In other words, there exists
$h\in\C[z,t]$ such that $h_t=f$ and $h_z=-g$.

Now consider $\varphi(r)\equiv r+x(\{r,h\}) \mod (x^2)$. Since
$\varphi(r)\in(r,x^2)$, the Poisson bracket $\{r,h\}$ is in the
ideal $(r)$. This implies that there exists a constant $c\in\C$
such that $h-c\in(r)$. To see this, note that $\diff h\wedge\diff
r$ is identically zero along the zero set $V(r)$ of $r$. Thus $h$
is locally constant as a function on $V(r)$ in a neighborhood of
every smooth point of $V(r)$. Since $r$ has no multiple
irreducible factor, the set of smooth points is dense. Since
$V(r)$ is connected, $h$ is constant along $V(r)$.

We may assume that the constant $c=0$. Thus $h=r\alpha$, with
$\alpha\in\C[z,t]$, and $\varphi(z)\equiv z+x(r\alpha)_t\mod
(x^2)$ and $\varphi(t)\equiv t-x(r\alpha)_z\mod (x^2)$.

It is easy to check that $\theta:\A_1(r)\to\C[z,t]$,
$\varphi\mapsto\alpha$ is a group homomorphism, whose kernel is
$\A_2$. We now prove that it is surjective. For any
$\alpha\in\C[z,t]$ define an automorphism
$\overline{\varphi}\in\Aut_RR[z,t]$ where $R=\C[x]/(x^2)$, given
by $\overline\varphi(z)=z+x(r\alpha)_t$ and
$\overline\varphi(t)=t-x(r\alpha)_z$. Note that the inverse of
$\overline\varphi$ is given by
$\overline\varphi^{-1}(z)=z-x(r\alpha)_t$ and
$\overline\varphi^{-1}(t)=t+x(r\alpha)_z$. Also,
$\overline\varphi$ is indeed an automorphism, and its Jacobian is
$1$. By a result of van den Essen, Maubach and V\'en\'ereau,
\cite{E-M-V}, there exists an automorphism $\varphi$ of
$\C[x][z,t]$ which projects to $\overline\varphi$. By
construction, $\varphi\in\A_1(r)$ and $\theta(\varphi)=\alpha$.
\end{proof}

\section{Extensions of automorphisms}

In this section, we will continue with the more general setting in
order to prove the lemma \ref{A1-ext} below. We then will apply
the lemma to the hypersurface $X$.

\begin{note} Let $r\in\C[z,t]$ be a polynomial with
no multiple irreducible factor, whose zero set is connected, and
let $F$ be any polynomial in $\C[x,z,t]$. We define
$P_{r,F}=x^2y+r+xF\in\C[x,y,z,t]$, and we let $X_{r,F}=V(P_{r,F})$.
Thus, for example, $X=X_{z^2+t^3,1}$, and
$Y=X_{z^2+t^3,(1+z^2+t^3+x)}$.
\end{note}

As in the proof of proposition \ref{aut-group},  for any
$\varphi\in\A_1(r)$, we can construct an endomorphism $\Phi$ of
$\C[x,y,z,t]=\C^{[4]}$ which induces a unique automorphism
$\tilde\varphi$ of $\C[X_{r,F}]$ as follows. $\Phi$ is an
extension of $\varphi$ where we determine $\Phi(y)$.  Suppose that
$\theta(\varphi)=\alpha$. Let $\beta=\{r,\alpha\}$. Then, it is
easily checked that $\varphi(r+xF)\equiv (1+x\beta)(r+xF)\mod
(x^2)$. Therefore, there exists a unique $G\in\C[x,z,t]$ such that
$\Phi(P_{r,F})=(1+x\beta)P_{r,F}$ if we pose
$\Phi(y)=(1+x\beta)y+G$. We will denote by $\tilde\varphi$ the
induced automorphism on $\C[X_{r,F}]$. In this way, $\A_1(r)$ can
be considered as a subgroup of $\Aut(\C[X_{r,F}])$. We will now
show that any such automorphism of $X_{r,F}$ lifts to an
automorphism of $\ba^4$. This is clear for the case that
$\beta=0$, since in this case, $\Phi$ is an automorphism of
$\C^{[4]}$. In particular, any automorphism of $\A_2$ induces an
automorphism of $X_{r,F}$ which extends. However, even if
$\beta\not=0$, we will show that by adding an appropriate multiple
of $P_{r,F}$, we can lift $\tilde\varphi$ to an automorphism of
$\C[x,y,z,t]$.

\begin{lem}\label{A1-ext}
Let $\varphi\in\A_1(r)$, then $\tilde\varphi$, the corresponding
automorphism of $\C[X_{r,F}]$,  lifts to an automorphism of
$\C^{[4]}$.
\end{lem}

\begin{proof} Let $\varphi\in \A_1(r)$, and suppose that $\theta(\varphi)=\alpha$.  Similarly to the method used  in \cite{MJ-P}, we create a family of endomorphisms of $\ba^4$ each one  restricting to an automorphism of a fiber of $P_{r,F}$.
Consider $c$ as a variable, and denote by $R_c$ the ring
$R_c=\C[x,c]/(x^2)$. Consider now the automorphism
$\overline{\phi}\in\Aut_{R_c}(R_c[z,t])$ given by
$\overline\phi(z)=\varphi(z)+xc\alpha_t$, and
$\overline\phi(t)=\varphi(t)-xc\alpha_z$. One checks easily that
the Jacobian of $\overline\phi$ is 1, and therefore, by the result
of van den Essen, Maubach and V\'en\'ereau, \cite{E-M-V}, there
exists an automorphism $\phi\in\Aut_{\C[c,x]}(\C[x,c][z,t])$ which
restricts to $\overline\phi$.

 For each $c\in\C$, denote by
$\varphi_c\in\Aut_{\C[x]}(\C[x][z,t])$ the automorphism defined by
$\varphi_c(z)=\phi(z)$ and $\varphi_c(t)=\phi(t)$, where $\phi(z)$
and $\phi(t)$ are viewed as polynomials with coefficients in
$\C[c]$. Note that $\varphi_c\in\A_1(r+c)$ and that the expression
for $\varphi_c$ depends polynomially on $c$.

For each $c\in\C$, we now construct, similarly to above, an
automorphism $\tilde{\varphi_c}$ on $\C[X_{r+c,F}]$. Note that
the expression for $\tilde{\varphi_c}$ depends polynomially on
$c$. By making a formal substitution of $c$ by $-P_{r,F}$, we
construct an automorphism $\Psi=\tilde{\varphi}_{(-P_{r,F})}$ of
$\C[x,y,z,t]$ which preserves the ideal $(P_{r,F})$. (See
\cite{MJ-P}, lemma 3.4). Note that $\Psi$ is a lift of the
automorphism $\tilde{\varphi_0}\in\Aut(\C[X_{r,F}])$. Also,
$\varphi_0$ and $\varphi$ are equivalent modulo $(x^2)$. More
precisely, $\varphi_0^{-1}\circ\varphi$ is an element of $\A_2$.
By the comment before the lemma, $\varphi_0^{-1}\circ\varphi $ induces an automorphism $\widetilde{\varphi_0^{-1}\circ\varphi}$ which lifts to an automorphism of $\C^{[4]}$. By the unicity of the extension of an element of $\A_1(r)$  to an automorphism of $\C[X_{r,F}]$, we have that $\widetilde{\varphi_0^{-1}\circ\varphi}=\tilde{\varphi_0}^{-1}\circ\tilde\varphi$. Since $\tilde{\varphi_0}$  also lifts to an automorphism of $\C^{[4]}$, the same is true for $\tilde{\varphi}$.
\end{proof}

\begin{thm}\label{X-ext} Every automorphism of $X=V(P)$ extends to an automorphism of $\ba^4$.
\end{thm}
\begin{proof}
The automorphism group of $X$ is isomorphic to
$\A=\A_1\rtimes\C^*$. The automorphisms in $\C^*$ extend, and, by
lemma \ref{A1-ext},  the automorphisms in $\A_1$ extend. Therefore
all automorphisms extend to  automorphism of $\ba^4$.
\end{proof}

\begin{example}
Consider the automorphism $\varphi$ of $\C[x,z,t]$ given by
$\varphi(x)=x$, $\varphi(z)=z+3xt^5 $ and
$\varphi(t)=t+2x(z+3xt^5)^3$. It is indeed an automorphism, since
it is the composition of two triangular automorphisms.

Also we can check that $\varphi$ is an element of $\A_1$. It is
obvious that $\varphi(x)=x$ and $\varphi\equiv\id\mod (x)$;  we
will now show that $\varphi(z^2+t^3+x)$ is in the ideal
$(x^2,z^2+t^3+x)$. Indeed, we find that there exists an element
$G\in\C[x,z,t]$ such that
$\varphi(z^2+t^3+x)=(z^2+t^3+x)+x(6zt^5+6t^2z^3)+x^2G$. This
yields : $\varphi(z^2+t^3+x)=(z^2+t^3+x)(1+6xzt^2)+x^2(G-6zt^2)$.

Thus $\varphi$ is an element of $\A_1$. To find the corresponding
automorphism of $\C[X]$, we extend $\varphi$ to the automorphism
$\tilde{\varphi}$ of $\C[X]$  where
 $\tilde\varphi(y)=(1+6xzt^2 )y-(G-6zt^2)$.

In order to lift this automorphism to an automorphism of
$\C^{[4]}$, we apply the procedure above.

We have that $\alpha=(t^3-z^2)/2$. We define, for each $c\in\C$, an automorphism $\varphi_c\in\A_1(z^2+t^3+c)$ as follows.
$\varphi_c(z)=z+3xt^2(t^3+c/2)$, and $\varphi_c(t)=t+2x\varphi_c(z)(\varphi_c(z)^2+c/2)$. More precisely, we have that $\varphi(z)\equiv z+x((z^2+t^3+c)\alpha)_t \mod (x^2)$, and $\varphi(t)\equiv t-x((z^2+t^3+c)\alpha)_z \mod (x^2)$.

Now, we can define for each $c\in\C$, an automorphism
$\tilde{\varphi_c}$ of $\C[X_{z^2+t^3+c,1}]$ if we pose
$\tilde\varphi_c(y)=(1+6xzt^2)y+G_c$, for a suitable polynomial
$G_c\in\C[x,z,t,c]$. Finally, to find the automorphism of
$\C[x,y,z,t]$ which is a lift of $\tilde\varphi$, we make a formal
substitution of $c$ by $-P$.
\end{example}

\begin{cor}\label{fix-origin} The origin $o=(0,0,0,0)\in X$ is fixed by all automorphisms of $X$.
\end{cor}
\begin{proof}
By proposition \ref{prop:extension} (ii), any automorphism of $X$
which extends to $\ba^4$ fixes the origin. By theorem \ref{X-ext},
all automorphisms of $X$ extend to $\ba^4$. \end{proof}

\begin{rem*} This corollary was first proven in collaboration with G. Freudenburg using a different method.
\end{rem*}

\section{Inequivalent hypersurfaces}

Consider now the two hypersurfaces $X=V(P)$ and $Y=V(Q)$. (As
before, $P=x^2y+z^2+x+t^3)$, and $Q=x^2y+(1+x)(z^2+x+t^3)$). We
know from theorem \ref{main thm} that as abstract varieties, $X$
and $Y$ are isomorphic. We now show the following result
\begin{thm} $X$ and $Y$ are inequivalent as hypersurfaces of $\C^4$.
\end{thm}
\begin{proof}
Suppose there were an automorphism $\Psi$ of $\ba^4$ such that
$\Psi(X)=Y$. Then $\Psi(o)\not=o$ by proposition
\ref{prop:extension} (i). Consider now the isomorphism  $\phi$
defined in theorem \ref{main thm} between $X$ and $Y$. Then
$(\Psi^{-1})|_Y\circ\phi$ is an automorphism of $X$ which does not
preserve $o$. This contradicts the corollary \ref{fix-origin}.
\end{proof}

It should be noted that, as another consequence of the description
of the automorphism group of $X$, we can show that all
automorphisms of $\C[Y]$ which fix the variable $x$ also extend to
automorphisms of $\ba^4$.

This is the case, since
$\Aut(Y)\cong\Aut(X)\cong\A=\A_1\rtimes\C^*$. The subgroup of
automorphisms which fix $x$ (for $X$ or for $Y$) corresponds via this isomorphism to the subgroup  $\A_1\rtimes
C_6\subset\A_1\rtimes\C^*$, where $C_6$ is the subgroup of the sixth roots of unity in
$\C^*$.  The automorphisms corresponding to elements of  $\A_1$ extend to $Y$ by lemma
\ref{A1-ext}, and the automorphisms corresponding to elements in $C_6$ extend to linear
automorphisms on $\ba^4$.

However, the action of $\C^*$ on $Y$ does not extend to an action
on $\ba^4$. More precisely, for any $\lambda\in\C$ such that
$\lambda^6\not=1$, the action of $\lambda\in\C^*$ on $X$,
conjugated by $\phi$ to give an automorphism of $Y$, does not
extend to an automorphism of $\ba^4$. To see this, note that the
action of $\lambda$ does not fix the line $x=z=t=0$. The only
fixed point on this line is the origin. However, by proposition
\ref{prop:extension}, (iii), the point $(0,-1,0,0)$ must be fixed.

\section{Remarks and some open questions}

\subsection{Locally nilpotent derivations on $\C[X]$}
We can give a complete description of the locally nilpotent
derivations of $\C[X]$.  We denote by $\LND(\C[X])$ the set of
locally nilpotent derivations of $\C[X]$.  Denote by
$\LND_x(\C[x][z,t])$ the $\C[x]$-module of locally nilpotent
derivations of $\C[x,z,t]$ having $x$ in the kernel. If $\partial$
is  a locally nilpotent derivation on $\C[X]$, it restricts to a unique locally
nilpotent derivation of $\LND_x(\C[x][z,t])$.
\begin{prop} $\LND(\C[X])=x^2(\LND_x(\C[x][z,t]))$.
\end{prop}
\begin{proof} If $\partial=x^2\partial_0$ is an element of $x^2(\LND_x(\C[x][z,t])$, then one can extend it to a locally nilpotent derivation on $\C[X]$ by posing $\partial(y)=-\partial_0(z^2+t^3)$. For the converse, if $\partial$ is a locally nilpotent derivation on $\C[X]$, then $\partial(z^2+t^3)=2z\partial(z)+3t^3\partial(t)\in (x^2)$. Consider the derivation $\overline\partial$ of $\C[z,t]$ defined by $\overline\partial(f)\equiv\partial(f)\mod(x)$. Then $\overline\partial$ induces an action of $(\C,+)$ on $\ba^2$ which stabilizes the cuspidal curve $z^2+t^3=0$. This implies that the action is trivial, and therefore that $\overline\partial=0$. In other words, there exists an element $\partial_1\in\LND_x(\C[x][z,t])$ satisfying $\partial=x\partial_1$. Now since  $\partial(z^2+t^3)\in(x^2)$, we have that $\partial_1(z^2+t^3)$ belongs to $(x)$; and the same argument proves that there exists $\partial_0$ such that $\partial_1=x\partial_0$.
\end{proof}

\subsection{Non-zero fibers of $P$ and $Q$}

\begin{prop}
The following hold:
\begin{itemize}
\item[a)] For every $c\in\mathbb{C}$, $V(Q-c)$ is isomorphic to
the hypersurface $V(F_c)=Z_c$ of $\mathbb{A}^{4}$  defined by the
equation
$$F_c=x^2y+z^2+(1+c)x+t^3-c=0.$$

\item[b)] For every $c\in\C\setminus\{-1, 0\}$ and every
$c'\in\C\setminus\{0\}$, $Z_c$ and $V(P-c')$  are isomorphic as
abstract affine varieties.
\end{itemize}
\end{prop}

\begin{proof}
Recall that by definition \[ V(Q-c)\simeq{\rm
Spec}\left(\mathbb{C}\left[x,y,z,t\right]/\left(Q-c\right)\right)={\rm
Spec}\left(\mathbb{C}\left[x,y,z,t\right]/\left(x^{2}y+\left(1+x\right)\left(z^{2}+t^{3}+x\right)-c\right)\right).\]

We claim that the endomorphisms \[
\begin{cases}
\Phi_{c}: & \left(x,y,z,t\right)\mapsto\left(x,\left(1-x\right)y-z^{2}-x-t^{3},z,t\right)\\
\Psi_{c}: &
\left(x,y,z,t\right)\mapsto\left(x,\left(1+x\right)y+c,z,t\right)\end{cases}\]
of $\mathbb{A}^{4}$ restricts respectively to isomorphisms
$\phi_{c}:V(Q-c)\stackrel{\sim}{\rightarrow}Z_{c}$ and
$\psi_{c}:Z_{c}\stackrel{\sim}{\rightarrow}V(Q-c)$ which are
inverse to each other. Indeed, one checks that
$\Phi_{c}^{*}(F_{c})=\left(1-x\right)\left(Q-c\right)$ whereas
$\Psi_{c}^{*}\left(Q-c\right)=\left(1+x\right)F_{c}$ so that
$\Phi_{c}$ and $\Psi_{c}$ induce morphisms
$\phi_{c}:V(Q-c)\rightarrow Z_{c}$ and $\psi_{c}:Z_{c}\rightarrow
V(Q-c)$ respectively. Since $\phi_{c}$ and $\psi_{c}$ are
morphisms of schemes over $\mathbb{A}^{3}={\rm
Spec}\left(\mathbb{C}\left[x,z,t\right]\right)$, the identities
$(\Phi_{c}^{*}\circ\Psi_{c}^{*})(y)=y-\left(Q-c\right)$ and
$(\Psi_{c}^{*}\circ\Phi_{c}^{*})(y)=y-F_{c}$ guarantee that their
are inverse isomorphisms. This proves assertion a).

Now, note that $V(P-c)\simeq V(P-1)$ for every
$c\in\C\setminus\{0\}$. Indeed, if $c\in\C^*$, consider the
automorphism of $\mathbb{A}^4$ defined by
$\left(x,y,z,t\right)=\left(a^{6}x,a^{-6}y,a^{3}z,a^{2}t\right)$,
for a suitable constant $a\in\C$ such that $a^{-6}=c$. This
automorphism maps $V(P-c)$ isomorphically onto $V(P-1)$.

Finally, assertion b) follows from the fact that for every
$c\in\mathbb{C}\setminus\left\{-1\right\}$, the automorphism of
$\mathbb{A}^{4}$ defined by
$\left(x,y,z,t\right)\mapsto\left(\left(1+c\right)x,\left(1+c\right)^{-2}y,z,t\right)$
 maps $V(P-c)$
isomorphically onto $Z_{c}$.
\end{proof}

 Together with the above discussion, this result motivates the following
 question.

\begin{question}\label{question}
Are the subvarieties of $\mathbb{A}^{4}={\rm
Spec}\left(\mathbb{C}\left[x,y,z,t\right]\right)$ defined by the
equations \[ x^{2}y+z^{2}+x+t^{3}+1=0\qquad\textrm{and}\qquad
x^{2}y+z^{2}+t^{3}+1=0\] isomorphic ?
\end{question}

\begin{rem}
This question has an affirmative answer in holomorphic category. Indeed, one can easily check that the analytic automorphism of $\mathbb{A}^4$ defined by $$(x,y,z,t)\mapsto(x,y+1-\frac{e^x-1-x}{x^2}(z^2+t^3),e^{x/2}z,e^{x/3}t)$$ \noindent induces an isomorphism between the hypersurfaces $V(Q+1)$ and $V(P+1)$.
\end{rem}

\subsection{Holomorphic and stable equivalence }

\indent\newline\noindent Recall that two  closed algebraic smooth
subvarieties $X$ and $Y$ of $\mathbb{A}^{4}$ isomorphic as
abstract algebraic varieties are called holomorphically equivalent
 if there exists a biholomorphism of $\mathbb{A}^{4}$
restricting to a biholomorphism between $X$ and $Y$ considered as
complex manifolds. Similarly, we say that $X$ and $Y$ are stably
equivalent  if there exist $n\in\N$ and an algebraic automorphism of
$\mathbb{A}^{4+n}$ restricting to an isomorphism
between $X\times\mathbb{A}^{n}$ and $Y\times\mathbb{A}^{n}$. Here
we show the following.

\begin{prop}
The subvarieties $X$ and $Y$ of $\mathbb{A}^{4}={\rm
Spec}\left(\mathbb{C}\left[x,y,z,t\right]\right)$ defined by the
equations \[ P=x^{2}y+z^{2}+x+t^{3}=0\qquad\textrm{and}\qquad
Q=x^{2}y+\left(1+x\right)\left(z^{2}+x+t^{3}\right)=0\] are
holomorphically equivalent and stably equivalent.
\end{prop}
\begin{proof}
By virtue of theorem \ref{main thm}, $X$ and $Y$ are isomorphic as
abstract algebraic varieties. Holomorphic equivalence follows from
the observation that the map
$\zeta:\mathbb{A}^{4}\rightarrow\mathbb{A}^{4}$ defined by \[
\zeta\left(x,y,z,t\right)=\left(x,e^{x}y+x^{-2}\left(e^{x}-1-x\right)\left(z^{2}+x+t^{3}\right),z,t\right)\]
is a biholomorphism of $\mathbb{A}^{4}$ which maps $X$
isomorphically onto $Y$. Moreover, remark that $\zeta$ can be
viewed as a holomorphic extension of the isomorphism $\phi:X\to Y$
defined in theorem \ref{main thm}. Indeed, we have
$\zeta(x,y,z,t)=\left(x,\left(1+x\right)y+x^{-2}\left(e^{x}-1-x\right)P,z,t\right)$.

For stable algebraic equivalence, we consider the
$\mathbb{C}\left[x\right]$-endomorphism $\Psi^{*}$ of
$\mathbb{C}\left[x,y,z,t,w\right]$ defined by \[
\begin{cases}
y & \mapsto y+1+x^{-2}\left(\left(\left(1+\frac{1}{2}x\right)z+x^{2}w\right)^2-\left(1+x\right)z^2\right)+x^{-2}\left(\left(\left(1+\frac{1}{3}x\right)t+x^{2}w\right)^3-\left(1+x\right)t^3\right)\\
z & \mapsto\left(1+\frac{1}{2}x\right)z+x^{2}w\\
t & \mapsto\left(1+\frac{1}{3}x\right)t+x^{2}w\\
w &
\mapsto-\frac{3}{4}z+\frac{2}{9}t+\left(1-\frac{5}{6}x\right)w\end{cases}\]
By definition, $\Psi^{*}$ restricts to a linear
$\mathbb{C}\left[x\right]$-endomorphism $\tilde{\Psi}^{*}$ of
$\mathbb{C}\left[x,z,t,w\right]$, which is an automorphism since
the Jacobian is invertible :
 \[
\left(\begin{array}{ccc}
\left(1+\frac{1}{2}x\right) & 0 & x^{2}\\
0 & \left(1+\frac{1}{3}x\right) & x^{2}\\
-\frac{3}{4} & \frac{2}{9} &
\left(1-\frac{5}{6}x\right)\end{array}\right)\in{\rm
GL}_{3}\left(\mathbb{C}\left[x\right]\right).\] Since
$\Psi^{*}(y)$ depends triangularly on the variables $x$, $z$, $t$
and $w$, this implies in turn that $\Psi^{*}$ is a
$\mathbb{C}\left[x\right]$-automorphism of
$\mathbb{C}\left[x,y,z,t,w\right]$. Now one checks easily that
$\Psi^{*}P=Q$, which completes the proof.
\end{proof}

\begin{rem}
In the proof above, we have shown that the isomorphism $\phi:X\to
Y$ -- which cannot be extended to an algebraic automorphism of
$\mathbb{A}^{4}$ (theorem \ref{main thm}) --, can be extended to a
holomorphic automorphism of $\mathbb{A}^{4}$. In particular, there is  no topological obstruction to extending
the isomorphism $\phi$ to an automorphism.
\end{rem}

\section{Locally nilpotent derivations on the cylinder over the
Koras-Russell threefold}

Recently, the first author proved that the Makar-Limanov invariant
of the cylinder over the Koras-Russell threefold is trivial
\cite{Dubouloz08}. The idea of the proof is as follows. Let
$X^0=X\setminus V(t)$. Now consider the polynomial
$P_1=xy+z^2+t^3+x$, and let $X_1=V(P_1)$, and $X_1^0=X_1\setminus
V(t)$.  It is shown that the cylinders $X^0\times\C$ and
$X_1^0\times\C$ are isomorphic. Then one uses the fact that the
Makar-Limanov invariant of $X_1^0$ is trivial to show that the
Makar-Limanov invariant of $X^0\times\C$ is trivial, and this
implies the result. 

We denote by, $A=\C[X]$ is the coordinate ring of $X$, by $B=\C[X_1]$ the coordinate ring of $X_1$, by $A_t$ the coordinate ring of $X^0$, and by $B_t$ the coordinate ring of $X_1^0$. 

In this section we will construct an explicit isomorphism, and
then we will obtain a locally nilpotent derivation $\partial$ on
the coordinate ring $A[w]$ of $X\times\ba^1$ such that
$\partial(x)\neq0$.

\begin{prop}\label{prop:iso-entre-algebres}
The algebraic endomorphism $\Phi$ of $\C[x,y,z,t^{\pm 1},w]$
defined by
$$\left\{\begin{array}{l}
\Phi(x)=x \\
\Phi(y)=xy-xw^2-2zw \\
\Phi(z)=z+xw \\
\Phi(t)=t\\
\Phi(w)=2w+yz+3xyw-3zw^2-xw^3
\end{array} \right. $$
\noindent induces an isomorphism
$$\phi:A_t[w]=\C[x,y,z,t^{\pm 1},w]\slash(x^2y+z^2+x+t^3)
\stackrel{\sim}{\rightarrow}B_t[w]=\C[x,y,z,t^{\pm
1},w]\slash(xy+z^2+x+t^3)$$

\noindent whose inverse isomorphism $\phi^{-1}:B_t[w]\to A_t[w]$
is induced by the endomorphism of $\C[x,y,z,t^{\pm 1},w]$ defined
by
$$\left\{\begin{array}{l}
\Psi(x)=x \\
\Psi(y)=-\frac{1}{t^3}(y+y^2+wz)-\frac{1}{4t^6}(yz-xw)^2 \\
\Psi(z)=z-\frac{1}{2t^3}x(yz-xw) \\
\Psi(t)=t\\
\Psi(w)=\frac{1}{2t^3}(yz-xw).
\end{array} \right. $$
\end{prop}

\begin{proof}
Recall that $P$ denotes the polynomial $P=x^2y+z^2+x+t^3$. Let $S$
be the polynomial defined by $S=xy+z^2+x+t^3$.

$\Phi$ and $\Psi$ induce
 well-defined algebraic morphisms $\phi:A_t[w]\to B_t[w]$ and
 $\psi:B_t[w]\to A_t[w]$. Indeed one checks $\Phi(S)=P$ and $\Psi(P)=(1-xyt^{-3})S$.

 Since $\phi$ and $\psi$ are $\C[x,t^{\pm 1}]$-morphisms, the following equalities prove that $\phi$ and $\psi$ are inverse
 isomorphisms.
$$\Psi\circ\Phi(z)=z\quad;\quad\Psi\circ\Phi(y)=y-\frac{y}{t^3}S\quad;\quad
\Psi\circ\Phi(w)=w+\frac{xyw-y^2z-t^3w}{t^6}S;$$
$$\Phi\circ\Psi(w)=w-\frac{1}{t^3}P\quad;\quad \Phi\circ\Psi(z)=z+\frac{xw}{t^3}P\quad;\quad \Phi\circ\Psi(y)=y-\frac{y-w^2}{t^3}P-\frac{w^2}{t^6}P^2. $$
\end{proof}

\begin{prop}\label{prop:lnd}
Let $\Delta$ be the locally nilpotent derivation on
$\C[x,y,z,t^{\pm 1},w]$ defined by
$$\Delta=t^6\left(-2z\frac{\partial}{\partial x}+(y+1)\frac{\partial}{\partial z}\right).$$ Then, the derivation
$\partial$ of $\C[x,y,z,t^{\pm 1},w]$
 defined by
$$\partial=(\Phi\circ\Delta\circ\Psi)(x)\frac{\partial}{\partial x}+(\Phi\circ\Delta\circ\Psi)(y)\frac{\partial}{\partial y}+(\Phi\circ\Delta\circ\Psi)(z)\frac{\partial}{\partial z}+(\Phi\circ\Delta\circ\Psi)(w)\frac{\partial}{\partial w}$$
\noindent induces a locally nilpotent derivation on
$A[w]=\C[x,y,z,t,w]\slash(x^2y+z^2+x+t^3)$ which does not
contain the variable $x$ in its kernel.
\end{prop}

\begin{proof}
Note first that $\partial(x),\partial(y),\partial(z)$ and
$\partial(w)$ are all elements of $\C[x,y,z,t,w]$ and so
$\partial$ restricts to a well-defined  derivation on $\C[x,y,z,t,w]$. (For example, one checks that $\partial(x)=-2t^6(z+xw)$).

Secondly, notice that, since $\Delta(xy+z^2+x+t^3)=0$, $\Delta$
induces a locally nilpotent derivation  on $B_t[w]=\C[x,y,z,t^{\pm
1},w]\slash(xy+z^2+x+t^3)$. Therefore,
 in light of proposition \ref{prop:iso-entre-algebres}, on can conclude that
 $\partial$ induces a locally nilpotent derivation $\tilde\partial$ on
 $A_t[w]=\C[x,y,z,t^{\pm 1},w]\slash(x^2y+z^2+x+t^3)$.
 
The inclusion $\C[x,y,z,t,w]\subset\C[x,y,z,t^{\pm 1},w]$ induces an inclusion of $A[w]$ in $A_t[w]$. More precisely,   let $\pi:\C[x,y,z,t^{\pm 1},w]\to A_t[w]$ be the canonical projection. Since $P$ is prime in $\C[x,y,z,t,w]$, one can identify $A[w]$ with the image $\pi(\C[x,y,z,t,w])$ of the subalgebra $\C[x,y,z,t,w]$. We have that $\pi\circ\tilde\partial=\partial\circ\pi$, and therefore $\tilde\partial$ restricts to a locally nilpotent deriviation on $A[w]=\pi(\C[x,y,z,t,w])$.
\end{proof}

\begin{cor}
The Makar-Limanov invariant of $A[w]$ is trivial.
\end{cor}

\begin{proof}
Let $\partial_1$ and $\partial_2$ be the locally nilpotent
derivations on $A[w]$ defined by
$$\partial_1=2z\frac{\partial}{\partial y}-x^2\frac{\partial}{\partial z}\quad\text{and}\quad\partial_2=3t^2\frac{\partial}{\partial y}-x^2\frac{\partial}{\partial t}  .$$
We have $\Ker(\partial_1)\cap\Ker(\partial_2)=\C[x]$. Thus,
$\ML(A[w])\subset\C[x]$. Now, proposition \ref{prop:lnd} allows us
to conclude that $\ML(A[w])=\C$.
\end{proof}

\bibliographystyle{plain}

\bibliography{KRreferences}

\begin{thebibliography}{10}

\bibitem{A-M}
Shreeram~S. Abhyankar and Tzuong~Tsieng Moh.
\newblock Embeddings of the line in the plane.
\newblock {\em J. Reine Angew. Math.}, 276:148--166, 1975.

\bibitem{Dimca}
{A.} Dimca.
\newblock Hypersurfaces in $\mathbb{C}^n$ diffeomorphic to $\mathbb{R}^{4n-2}$.
\newblock Max-Planck Institut Bonn, preprint, 1990.

\bibitem{Dubouloz08}
{A}drien {D}ubouloz.
\newblock {T}he cylinder over the {K}oras-{R}ussell cubic threefold has a
  trivial {M}akar-{L}imanov invariant.
\newblock to appear in Trans. Groups, 2009.

\bibitem{D-P}
{A}drien {D}ubouloz and {P}ierre-{M}arie {P}oloni.
\newblock {O}n a class of {D}anielewski surfaces in affine 3-space.
\newblock to appear in J. of Algebra, 2009.

\bibitem{Freudenburg}
Gene Freudenburg.
\newblock {\em Algebraic theory of locally nilpotent derivations}, volume 136
  of {\em Encyclopaedia of Mathematical Sciences}.
\newblock Springer-Verlag, Berlin, 2006.
\newblock Invariant Theory and Algebraic Transformation Groups, VII.

\bibitem{F-MJ}
Gene Freudenburg and Lucy Moser-Jauslin.
\newblock Embeddings of {D}anielewski surfaces.
\newblock {\em Math. Z.}, 245(4):823--834, 2003.

\bibitem{Hartshorne}
Robin Hartshorne.
\newblock {\em Algebraic geometry}.
\newblock Springer-Verlag, New York, 1977.
\newblock Graduate Texts in Mathematics, No. 52.

\bibitem{K-ML}
Sh. Kaliman and L.~Makar-Limanov.
\newblock A{K}-invariant of affine domains.
\newblock In {\em Affine algebraic geometry}, pages 231--255. Osaka Univ.
  Press, Osaka, 2007.

\bibitem{K-Z}
Sh. Kaliman and M.~Zaidenberg.
\newblock Affine modifications and affine hypersurfaces with a very transitive
  automorphism group.
\newblock {\em Transform. Groups}, 4(1):53--95, 1999.

\bibitem{Kaliman}
Shulim Kaliman.
\newblock Extensions of isomorphisms between affine algebraic subvarieties of
  {$k\sp n$} to automorphisms of {$k\sp n$}.
\newblock {\em Proc. Amer. Math. Soc.}, 113(2):325--334, 1991.

\bibitem{ML}
L.~Makar-Limanov.
\newblock On the hypersurface {$x+x\sp 2y+z\sp 2+t\sp 3=0$} in {${\bf C}\sp 4$}
  or a {${\bf C}\sp 3$}-like threefold which is not {${\bf C}\sp 3$}.
\newblock {\em Israel J. Math.}, 96(, part B):419--429, 1996.

\bibitem{ML2001}
L.~Makar-Limanov.
\newblock On the group of automorphisms of a surface {$x\sp ny=P(z)$}.
\newblock {\em Israel J. Math.}, 121:113--123, 2001.

\bibitem{MJ-P}
Lucy Moser-Jauslin and Pierre-Marie Poloni.
\newblock Embeddings of a family of {D}anielewski hypersurfaces and certain
  {$\bold C\sp +$}-actions on {$\bold C\sp 3$}.
\newblock {\em Ann. Inst. Fourier (Grenoble)}, 56(5):1567--1581, 2006.

\bibitem{Poloni}
{P}ierre-{M}arie {P}oloni.
\newblock {\em {S}ur les plongements des hypersurfaces de {D}anielewski}.
\newblock PhD thesis, {U}niversit{\'e} de {B}ourgogne, 06 2008.

\bibitem{S-Y}
Vladimir Shpilrain and Jie-Tai Yu.
\newblock Embeddings of hypersurfaces in affine spaces.
\newblock {\em J. Algebra}, 239(1):161--173, 2001.

\bibitem{Srinivas}
V.~Srinivas.
\newblock On the embedding dimension of an affine variety.
\newblock {\em Math. Ann.}, 289(1):125--132, 1991.

\bibitem{Suzuki}
Masakazu Suzuki.
\newblock Propri\'et\'es topologiques des polyn\^omes de deux variables
  complexes, et automorphismes alg\'ebriques de l'espace {${\bf C}\sp{2}$}.
\newblock {\em J. Math. Soc. Japan}, 26:241--257, 1974.

\bibitem{E-M-V}
Arno van~den Essen, Stefan Maubach, and St{\'e}phane V{\'e}n{\'e}reau.
\newblock The special automorphism group of {$R[t]/(t\sp m)[x\sb 1,\dots,x\sb
  n]$} and coordinates of a subring of {$R[t][x\sb 1,\dots,x\sb n]$}.
\newblock {\em J. Pure Appl. Algebra}, 210(1):141--146, 2007.

\end{thebibliography}

\end{document}